\documentclass[12pt]{amsart}   	
\usepackage{geometry}                		
\usepackage{graphicx}				
\usepackage{amssymb}
\usepackage{mathtools}

\usepackage{amsthm,amssymb,amsmath,amsfonts,extarrows}
\usepackage[all,cmtip,2cell]{xy}

\usepackage[colorlinks=true]{hyperref}

\usepackage{multirow}

\usepackage[dvipsnames]{xcolor}

\usepackage{comment}
\usepackage{tikz-cd}

\newtheorem{manualtheoreminner}{Theorem}
\newenvironment{manualtheorem}[1]{%
  \renewcommand\themanualtheoreminner{#1}%
  \manualtheoreminner
}{\endmanualtheoreminner}

\newtheorem{theorem}{Theorem}[section]
\newtheorem{lemma}[theorem]{Lemma}
\newtheorem{cor}[theorem]{Corollary}

\newtheorem{prop}[theorem]{Proposition}

\theoremstyle{definition}

\newtheorem{remark}[theorem]{Remark}

\numberwithin{equation}{theorem}

\newcommand{\PP}{\mathbb{P}}
\newcommand{\QQ}{\mathbb{Q}}
\newcommand{\RR}{\mathbb{R}}

\newcommand{\ZZ}{\mathbb{Z}}

\newcommand{\hfal}{h_{\mathrm{Fal}}}
\newcommand{\cNT}{c_{\mathrm{NT}}}

\DeclareMathOperator{\Sp}{Sp}

\DeclareMathOperator{\Stab}{Stab}

\DeclareMathOperator{\Vol}{Vol}

 \newcommand{\Tangli}[1]{} 

\begin{document}
\title{Uniform Mordell-Lang Plus Bogomolov}
\author{Tangli Ge}

\address{Department of Mathematics\\
Brown University\\
Box 1917\\
Providence, RI 02912\\
U.S.A.}
\email{\url{tangli_ge@brown.edu}}

\thanks{Work supported in part by funds from NSF grant
   DMS-2100548 and DMS-1759514}

\begin{abstract}
	In this paper, we prove a uniform version of Poonen's "Mordell-Lang Plus Bogomolov" theorem \cite{PoonenMLB}, based on Vojta's method. We mainly generalize R\'emond's work on large points to allow an extra $\epsilon$-neighborhood. The part on small points follows from \cite{GGK}. 
	\end{abstract}
    
\maketitle


\section{Introduction}
Throughout, we work over an algebraic closure $\bar\QQ$ of the rationals $\QQ$. Let $A$ be an abelian variety defined over $\bar\QQ$, and let $A(\bar\QQ)$ be the group of algebraic points on $A$. Let $\Gamma$ be a finitely generated subgroup of $A(\bar\QQ)$. The \emph{division group} $\Gamma'$ of $\Gamma$ is defined as
\[
\Gamma':=\{x\in A(\bar\QQ)|\text{ there exists }n\geq1\text{ such that } nx\in\Gamma\}.
\]

 The Mordell-Lang conjecture, proved in the case of abelian varieties by Faltings \cite{Faltings83,Faltings}, Vojta \cite{Vojta91} and Raynaud \cite{Raynaud:ML}, states that, if an integral subvariety $X$ of $A$ is not a \emph{coset}, i.e. a translate of an abelian subvariety by a closed point, then the intersection $X(\bar\QQ)\cap\Gamma'$ is not Zariski dense in $X$. 
 
 Now assume moreover that $A$ is equipped with a N\'eron--Tate height $\hat h:A(\bar\QQ)\rightarrow \RR_{\geq0}$. The Bogomolov conjecture proved by Ullmo \cite{Ullmo98} and S. Zhang \cite{Zhang98}, states that if $X$ is not a \emph{torsion coset} which is a translate of an abelian subvariety by a torsion point, then there is some $\epsilon>0$ such that the set of small points 
 \[
\{P\in X(\bar\QQ):\hat h(P)<\epsilon\}
\]
is not Zariski dense in $X$.

Later, Poonen \cite{PoonenMLB} (also by S. Zhang in \cite{Zhang00}) shows that the Mordell-Lang conjecture and the Bogomolov conjecture, together with an equidistribution theorem, imply a stronger result, which trivially encompasses both conjectures.  For $\epsilon>0$, define the \emph{$\epsilon$-neighborhood} $\Gamma'_\epsilon$ of the division group $\Gamma'$ as 
\[
\Gamma'_\epsilon:=\{\gamma+z:\gamma\in\Gamma', z\in A(\bar\QQ), \hat h(z)<\epsilon\}.
\]
What Poonen and Zhang proved  is the following:
{\begin{theorem}[Poonen--Zhang]
	 Let $X\subseteq A$ be an integral subvariety which is not a coset, and let $\Gamma\leq A(\bar\QQ)$ be a finitely generated subgroup. Then there is some $\epsilon=\epsilon(X,A,\Gamma)>0$ such that the intersection $X(\bar\QQ)\cap \Gamma'_\epsilon$ is not Zariski dense in $X$.
\end{theorem}}

A standard recursive application of the above theorem leads to the following equivalent version:
{\begin{manualtheorem}{1.1'}[Poonen--Zhang] \label{poonen}
	Let $X\subseteq A$ be an integral subvariety, and let $\Gamma\leq A(\bar\QQ)$ be a finitely generated subgroup. Then there is some $\epsilon=\epsilon(X,A,\Gamma)$ such that the intersection $X(\bar\QQ)\cap\Gamma_\epsilon'$ is a finite union of $Y_i(\bar\QQ)\cap\Gamma_\epsilon'$, where $\{Y_i\}_i$ is a finite set of cosets in $X$.
\end{manualtheorem}}

Define the \emph{special locus} of $X$, denoted by $\Sp(X)$, as the union of positive-dimensional cosets in $X$, which is Zariski closed as shown by Kawamata \cite{Kawamata}.  Denote the open complement by $X^\circ:=X-\Sp(X)$. Then all cosets in $X^\circ$ are just points and Theorem \ref{poonen} implies the finiteness of the set $X^\circ(\bar\QQ) \cap \Gamma'_\epsilon$. 

A motivating question for this paper is: \emph{can we choose $\epsilon$ above to be independent of the choice of $\Gamma$?} The answer is yes and indeed we can get a more uniform result, combining the uniform Mordell-Lang conjecture and the uniform Bogomolov conjecture proved in \cite{GGK}. 

Let $L$ be a symmetric (i.e. $[-1]^*L\cong L$) ample line bundle on $A$ which induces the associated N\'eron-Tate height $\hat h=\hat h_L:A(\bar\QQ)\rightarrow\RR_{\geq0}$. Fix the following notations:
 \[r:=\dim X,\,g:=\dim A ,\,d:=\deg_L X ,\, l:=\deg_L A.\]

The main result of this paper is the following:

\begin{theorem}\label{uniform_poonen}
	There exist positive constants $\epsilon=\epsilon(r,g,d)$ and $c=c(r,g,d)$\footnote{See the remark \ref{remark1}(2) below.} with the following property. For any abelian variety $(A,L)$, any integral subvariety $X\subseteq A$, and any finitely generated subgroup $\Gamma\leq A(\bar\QQ)$ of rank $\rho$, we have
	\[
	\# ( X^\circ(\bar\QQ)\cap \Gamma'_\epsilon ) \leq c^{1+\rho}.
	\]
\end{theorem}

\begin{remark}\label{remark1}
(1) Expressions such as $c=c(r,g,d)$ mean that the constant $c$ only depends on $r,g,d$. 

(2) We can actually prove a stronger version with $\epsilon$ replaced by a uniform multiple of certain normalized Faltings height of $A$; see \S\ref{section_further comments} for details.

(3) The degree $l$ of $A$ only appears in the middle of the proof, which is shown to become unnecessary, mainly because we can pass to the case where $X$ generates $A$. The dependence on $r=\dim X$ in the above theorem can also be removed easily, by simply taking
		\[
	\epsilon(g,d):=\min_{0\leq r\leq g} \{\epsilon(r,g,d)\},\quad c(g,d):=\max_{0\leq r\leq g}\{c(r,g,d)\}.
		\] Though redundant in the result, the index $r$ actually shows its importance in the inductive arguments later and we decide to keep this stratification.\Tangli{The index $r$ in my opinion is very important so I add a footnote to the statement of the above theorem instead. }

(4) The functions $\epsilon(r,g,d), c(r,g,d)$ are constructed in an increasing lexicographic order which we now describe. The set $\{(r,g,d)\}$ is totally ordered by the following rule: $(r_1,g_1,d_1)<(r_2,g_2,d_2)$ if either $r_1<r_2$; or $r_1=r_2, g_1<g_2$; or $r_1=r_2,g_1=g_2,d_1<d_2$. Then $\epsilon(r_1,g_1,d_1), c(r_1,g_1,d_1)$ are defined before $\epsilon(r_2,g_2,d_2), c(r_2,g_2,d_2)$ if $(r_1,g_1,d_1)<(r_2,g_2,d_2)$.

(5) The result of the above theorem is weaker if we decrease $\epsilon$ or increase $c$, but since our goal is to prove the existence, we will freely weaken the results, to ease our notations. By (3), without loss of generality we can and we do always make the following assumptions
\begin{itemize}
	\item $\epsilon$ decreases in all three variables;
	\item $c$ increases in all three variables;
	\item $c(r,g,d_1+d_2)\geq c(r,g,d_1)+c(r,g,d_2)$.
\end{itemize}  Then we may use the result even when $X$ is not irreducible or equidimensional. For example, if $X=X_1\cup X_2$ with $X_1,X_2$ irreducible of dimension $r$ and degree $d_1,d_2$ respectively, then with $\epsilon=\epsilon(r,g,d_1+d_2)$, by the third bullet point above we have 
\[
\begin{split}
    \# (X^\circ(\bar\QQ)\cap \Gamma'_\epsilon) &\leq\#(X_1^\circ(\bar\QQ) \cap \Gamma'_\epsilon) +\#(X_2^\circ(\bar\QQ)\cap \Gamma'_\epsilon)\\
    &\leq c(r,g,d_1)^{1+\rho}+c_2(r,g,d_2)^{1+\rho}\leq c(r,g,d_1+d_2)^{1+\rho}.
\end{split}
\]
\end{remark}

The proof of Theorem \ref{uniform_poonen} is based on Vojta's method, which has a dichotomy of large and small points in terms of their N\'eron--Tate heights. For the small points, we invoke a result of our previous work \cite{GGK} joint with Gao and K\"uhne, on a version of the uniform Bogomolov conjecture called the \emph{New Gap Principle}; see \S\ref{section_small_points}. The main part of this paper is to generalize the work of R\'emond \cite{remond-decompte, remond-inegalite} and David-Philippon \cite{DP07} on the large points to allow an extra $\epsilon$-neighborhood. Many ideas in the proof are borrowed from their work.

Remark that our proof is different from the proofs of Poonen and  Zhang. Their proofs assume the Mordellic part (the case when $\Gamma$ is finitely generated) of the Mordell--Lang conjecture, dive into the proof of the Bogomolov conjecture and argue by contradiction using a more careful analysis of the equidistribution of almost division points (see \cite[Theorem 1.1]{Zhang00}). It seems impossible to derive a uniform result from their approach. Our proof is closer to the approach of R\'emond \cite{Remond:MLpB}, in which he establishes the Mordell--Lang plus Bogomolov for semiabelian varieties without assuming equidistribution. Needless to say, the uniformity requires a more careful treatment.

Theorem \ref{uniform_poonen} can be improved to a slightly stronger version, in the flavor of Theorem \ref{poonen}, as follows, which is shown in \S\ref{section_finiteness_of_cosets}. 
{\begin{manualtheorem}{1.2'}\label{uniform_poonen2}
		There exist positive constants $\epsilon=\epsilon(g,d)$ and $c=c(g,d)$ with the following property. For any abelian variety $(A,L)$, any integral subvariety $X\subseteq A$, and any finitely generated subgroup $\Gamma\leq A(\bar\QQ)$ of rank $\rho$, the intersection $X(\bar\QQ)\cap\Gamma_\epsilon'$ is contained in the set of $\bar\QQ$-points of a union of at most $c^{1+\rho}$ many cosets in $X$.
\end{manualtheorem}}
Note that Theorem \ref{uniform_poonen2} encompasses both the uniform Mordell--Lang conjecture and the uniform Bogomolov conjecture proved in \cite{GGK}, but it does not follow directly from them.
\medskip
\begin{center}
    \textbf{Notations and Conventions}
\end{center}
\medskip

\begin{tabular}{ll}
    $\bar \QQ$   & an algebraic closure of $\QQ$.\\
    $A$          & an abelian variety over $\bar\QQ$.\\
    $A(\bar\QQ)$ & the group of $\bar\QQ$-points on $A$.\\
    $L$          & a symmetric ample line bundle on $A$.\\
    $\hat h$     & the N\'eron--Tate height $\hat h_L:A(\bar\QQ)\rightarrow                      \RR_{\geq0}$ associated to $L$.\\
    $\Gamma$     & a finitely generated subgroup of $A(\bar\QQ)$.\\
    $X$          & an integral (irreducible and reduced) subvariety of $A$.\\
    $X^\circ$    & the open complement of the special locus of $X$.\\
    $\rho$       & the rank of $\Gamma$ as an abelian group.\\
    $r$          & dimension of $X$.\\
    $g$          & dimension of $A$.\\
    $d$          & degree of $X$ with respect to $L$.\\
    $l$          & degree of $A$ with respect to $L$.\\
    $\langle\cdot,\cdot\rangle$ & the inner product on $A(\bar\QQ)\otimes_\ZZ                             \RR$ induced by $\hat h$.\\
    $|\cdot|$    &  magnitude associated to the inner product; so $|P|^2=\hat                       h(P)$ for $P\in A(\bar\QQ)$.\\
    $h(X)$       & height of $X$ with respect to the canonical adelic metric                 on $L$; see \S\ref{section_large_points}.\\
    $\hfal(A)$   & the stable Faltings height of $A$; see \cite{Faltings83}.
\end{tabular}

\medskip

\section{Technical Lemmas}
In this section, we prove several easy lemmas about Euclidean spaces. The reader shall feel free to skip this section and come back only when a lemma is invoked.
\begin{lemma}\label{lemma_anglebounds}
	Let $V$ be a real vector space with an inner product $\langle\cdot,\cdot \rangle$. For $\theta\in (0,\pi)$, there exists $0<\delta=\delta(\theta)<1$ with the following property: if $v_1,v_2\in V$ satisfy $|v_1|\geq c>0$ and $|v_2-v_1|\leq \delta \cdot c$, then the angle between $v_1,v_2$ is at most $\theta$.
\end{lemma}
\begin{proof}
	This is clear geometrically: if one looks out from the origin to $v_1$, then a small enough ball at $v_1$ will stay within the angle of vision. To prove it, take
	$\delta:=1-\frac{1}{2-\cos\theta}$.
	Then
	\[
		\frac{\langle v_1,v_2\rangle}{|v_1||v_2|}=\frac{|v_1|^2}{|v_1||v_2|}+\frac{\langle v_1,v_2-v_1 \rangle}{|v_1||v_2|}\geq 1-\frac{|v_1||v_2-v_1|}{|v_1||v_2|}\geq 1-\frac{\delta\cdot c}{c-\delta\cdot c}=\cos\theta.
	\]
	So the angle between $v_1,v_2$ is at most $\theta$.
\end{proof}

\begin{lemma}\label{lemma_cover_by_cones}
	Let $V$ be a finite-dimensional real vector space with an inner product $\langle\cdot,\cdot \rangle$, of dimension $\rho$. For any $c>1$, the vector space can be covered by at most $(1+\sqrt{8c})^{\rho}$ regions, such that for any $v_1,v_2$ in a same region, we have 
	\[
	\langle v_1,v_2\rangle \geq (1-\frac{1}{c})|v_1||v_2|.
	\]
\end{lemma}
\begin{proof}
	This is \cite[Corollaire 6.1]{remond-decompte}. We include the proof here for completeness.  

	Let $\theta:=\frac{1}{2}\arccos(1-\frac{1}{c})$. Denote by $B(x,R)$, the closed ball in $V$ of radius $R$ centered at $x$. We aim to find a covering of the \emph{unit sphere} $S$ by small pieces, such that the angle between any two points in a piece is at most $2\theta$. Then we can just cover $V$ by the one-sided cones spanned by these pieces, such that for any two vectors $v_1,v_2$ in a cone, \[\langle v_1,v_2\rangle\geq\cos(2\theta)|v_1||v_2|=(1-\frac{1}{c})|v_1||v_2|.\]
	
	The strategy is to cover $S$ first by small balls with centers on $S$, of radius $\sin\theta$. Each piece  cut out by the intersection of $S$ and a small ball has the required property. Indeed, the distance between any $v_1,v_2$ in a piece is at most $2\sin\theta$, which by an easy geometric argument in the isosceles triangle of side lengths $1,1,2\sin\theta$, implies that the angle between the vectors $v_1,v_2$ is at most $2\theta$.
	
	Now we construct our cover inductively as follows. Assume that we have chosen centers $x_1,..,x_n$, and that $S\not\subseteq\bigcup_{i=1}^n B(x_i,\sin\theta).$
	Then we pick $x_{n+1}$ in $S$ that is not yet covered. By the compactness of $S$, the procedure stops at a finite step. Assume eventually, we pick out centers $x_1,\ldots,x_N$, and no more points can be picked. 
	
	Notice that by our choice, the distance between any $x_i,x_j$ for $i\neq j$ is greater than $\sin\theta$. So in particular, we have
	$B(x_i,\frac{1}{2}\sin\theta)\cap B(x_j,\frac{1}{2}\sin\theta)=\emptyset.$
	Meanwhile,	$\bigcup_{i=1}^N B(x_i,\frac{1}{2}\sin\theta)\subseteq B(0,1+\frac{1}{2}\sin\theta).$
	Since $\Vol(B(0,R))=R^\rho\cdot \Vol(B(0,1))$, we get by comparing the volumes that $N\leq \left((1+\frac{1}{2}\sin\theta) / \frac{1}{2}\sin\theta\right)^\rho= (1+\sqrt{8c})^\rho,$ where the last equality follows from $\sin\theta=\sqrt{(1-\cos(2\theta))/2}$ and $\cos(2\theta)=1-\frac{1}{c}$.

\end{proof}

\begin{lemma}\label{lemma_cover_by_balls}
Let $V$ be a finite-dimensional real vector space with an inner product $\langle\cdot,\cdot\rangle$, of dimension $\rho$. Then a ball of radius $c$ can be covered by at most $(1+2c/c')^{\rho}$ balls of radius $c'$.
\end{lemma}
\begin{proof}
See also \cite[Lemme 6.1]{remond-decompte}.

	Same idea as the proof of Lemma \ref{lemma_cover_by_cones}. Exhaust points $x_1,...,x_N\in B(0,c)$ so that the distance between each pair is at least $c'$. Then $B(0,c)$ is covered by the balls of radius $c'$ centered at $c_i$ for $i=1,\ldots,N$. Also, we have $\bigcup_{i=1}^N B(x_i,\frac{1}{2}c')\subseteq B(0,c+\frac{1}{2}c')$. Comparing the volumes, we get
	$N\leq \left(  (c+\frac{1}{2}c')/\frac{1}{2}c' \right)^\rho=(1+2c/c')^\rho.$
\end{proof}

\medskip

\section{Large Points}\label{section_large_points}
In this section, we review results of R\'emond and apply them to our setting for a uniform treatment of large points. Many ideas are due to R\'emond and our job is to carefully take the $\epsilon$-neighborhood into account.

We assume $L$ to be moreover \emph{very ample} and induce a projectively normal closed immersion into some projective space, for this section. It is not much harder to show the general case, since if $L$ is ample then $L^{\otimes4}$ satisfies our assumption, see for example \cite{milneAV}. We have $h^0(A,L)=l/g!$ (recall that $l=\deg_L A$); see \cite[\S2]{GGK}. So $A$ can be embedded in $\PP^n$ with $n=l/g!-1$. 

The symmetric line bundle $L$ is endowed with a unique canonical adelic metric; the corresponding adelic metrized line bundle is denoted by $\bar L$. We can then define the height $h(X)$ of a subvariety $X\subseteq A$ with respect to $\bar L$ using the arithmetic intersection theory. See \cite[\S9]{CL21} for details.

\subsection{Inequalities}
There are two important constants that appear in the work of R\'emond, namely $\cNT$ and $h_1$. Let $h:\PP^n(\bar\QQ)\rightarrow \RR$ be the logarithmic Weil height. The first one $\cNT$ is a bound arising from the construction of N\'eron-Tate heights on the abelian variety. It satisfies $|\hat h(P)-h(P)|\leq \cNT$ for any $P\in A(\bar\QQ)$. The second one $h_1$ is the Weil height of the polynomials defining the addition and subtraction on the abelian variety. It is known that there is a constant $c'=c'(g,l)$ such that
\begin{equation}\label{bound_for_constants}
 \tag{*}
	\cNT,h_1\leq c'\max\{1,\hfal(A)\}.
\end{equation}
See \cite[(6.41)]{DP07} and \cite[(8.4)(8.7)]{DGH21} for details.

The N\'eron-Tate height $\hat h$ induces an inner product $\langle\cdot,\cdot \rangle$ on the vector space $A(\bar\QQ)\otimes_\ZZ \RR$, so that for any point $P\in A(\bar\QQ)$, we have $\langle P,P\rangle=\hat h(P)$. Write $|\cdot|$ for the induced norm.

Let $D_X:X^{r+1}\rightarrow A^{r}$ (recall that $r=\dim X$) be the morphism defined by $(x_0,\ldots,x_{r})\mapsto (x_1-x_0,\ldots,x_r-x_0)$. The following theorem by R\'emond gives explicit generalized Vojta's inequality and Mumford's inequality.
\begin{theorem}[R\'emond]
\label{theorem_inequalities}
	Let $X\subseteq A$ be an integral subvariety. There exist constants $c_1=c_1(r,g,d,l)>1$ and $c_2=c_2(r,g,d,l)>0$ with the following property. 
	\begin{enumerate}
\item (Vojta's inequality)  If $P_0,\ldots,P_{r}\in X^\circ (\bar\QQ)$ satisfy 
	\[
	\left\langle P_i, P_{i+1} \right\rangle \geq (1-\frac{1}{c_1})|P_i||P_{i+1}|\,\quad\text{and} \quad|P_{i+1}|\geq c_1|P_i|,
	\]for any $i$, then
	\[
	|P_0|^2\leq c_2\max\{1,h(X),\hfal(A)\}.
	\]
\item (Mumford's inequality) Suppose $(P_0,\ldots,P_r)$ is an isolated $\bar\QQ$-point in the fiber of $D_X:X^{r+1}\rightarrow A^{r}$. If 
	\[
	\left\langle P_0, P_i \right\rangle \geq (1-\frac{1}{c_1})|P_0||P_i|\,\quad\text{and} \quad\big||P_0|-|P_{i}|\big|\leq \frac{1}{c_1}|P_0|,
	\]for any $i$, then
	\[
	|P_0|^2\leq c_2\max\{1,h(X),\hfal(A)\}.
	\]
	\end{enumerate}
\end{theorem}
\begin{proof}
	Vojta's inequality is \cite[Th\'eor\`eme 1.1]{remond-inegalite} and Mumford's inequality is \cite[Proposition 3.4]{remond-decompte}. Note that we can remove $\cNT,h_1$ from them by \eqref{bound_for_constants}. 
\end{proof}

The following proposition is similar to R\'emond \cite[Proposition 3.3]{remond-decompte}. It will be used in Proposition \ref{prop_large_points} to ensure the extra condition in Mumford's inequality.
\begin{prop}\label{prop_degree_bound_for_isolation}
Let $\Xi\subseteq X^\circ(\bar\QQ)$ be a set of points.  Assume there is $P_0\in X(\bar\QQ)$ such that for any $P_1,\ldots,P_r\in \Xi$, the point $(P_0,P_1,\ldots,P_r)$ is not isolated in the fiber of $D_X: X^{r+1}\rightarrow A^r$. Then $\Xi$ is contained in the set of $\bar\QQ$-points of a Zariski closed subset $X'\subsetneq X$ with $\deg_L X'< d^{r+2}$.
\end{prop}
\begin{proof}
	If the dimension of the stabilizer $\Stab(X)$ of $X$ in $A$ is not $0$, then every point of $X$ is in a positive dimensional coset of $X$, whence $X^\circ=\emptyset$. So we assume $\Stab(X)$ is finite.
	
	Notice that the fiber of $D_X$ over $(P_1-P_0,\ldots,P_r-P_0)$ is
	\[
	\{(P_0+a,\ldots,P_r+a)\in X^{r+1}(\bar\QQ):a\in A(\bar\QQ)\},	\]
	which is isomorphic to $\bigcap_{i=0}^r (X-P_i)$, where $(P_0,\ldots,P_r)$ corresponds to $0$ under the isomorphism. Thus the condition that $(P_0,\ldots,P_r)$ is isolated in the fiber is equivalent to that the dimension of $\bigcap_{i=0}^r (X-P_i)$ at the origin is $0$. 
	
	Assume there exists $P_0\in X(\bar\QQ)$ such that $\dim_0 \bigcap_{i=0}^r (X-P_i)\neq 0$ for any choice of $P_1,\ldots,P_r\in \Xi$, where $\dim_0 Y$ stands for the dimension of $Y$ at $0$. Then we can use the greedy algorithm to pick out step by step $P_1,P_2,...,P_{r_0}$ for some $r_0<r$ with the following property: 
	\[ \dim_0 (X-P_0)> \dim_0 \bigcap_{i=0}^{1}(X-P_i)>\ldots > \dim_0\bigcap_{i=0}^{r_0}(X-P_i),\]
	and we cannot reduce the dimension at $0$ in one step any more. In other words, if we let $C_1,\ldots, C_s$ be the top-dimensional irreducible components passing through $0$ of $\bigcap_{i=0}^{r_0}(X-P_i)$, then for any $Q\in \Xi$, the translate $X-Q$ must contain some $C_i$, for some $i=i(Q)$. On the other hand, $C_i\subseteq X-Q$ if and only if $Q\in \bigcap_{a\in C_i(\bar\QQ)}(X-a).$ So we have
	\[
	\Xi\subseteq \bigcup_{i=1}^s\bigcap_{a\in C_{i}(\bar\QQ)}(X-a).
	\]
	Claim: $\bigcap_{a\in C_{i}(\bar\QQ)}(X-a)$ is contained in a proper subvariety $X_i$ of $X$ of degree at most $d^2$. Indeed, since $\Stab(X)$ is finite, there is some $a_i\in C_i(\bar\QQ)$ such that $X-a_i\neq X$. So simply take $X_i:=X\cap (X-a_i)$. On the other hand, $s\leq d^r$ since $\deg\bigcap_{i=0}^{r_0}(X-P_i)\leq d^{r}$. Let $X'$ be the union of $X_i$'s. Then $\deg X'\leq d^{r+2}$ and $\Xi\subseteq X'(\bar\QQ)$.
	\end{proof}

\subsection{Large points}
In the proof of Theroem \ref{uniform_poonen}, we will use induction on the dimension of $X$. To make things more clear, we extract the steps from the proof and make them propositions.
\begin{prop}\label{prop_large_points}
Assume Theorem \ref{uniform_poonen} holds for $\dim X\leq r-1$. In the case of $\dim X=r$, there exist positive constants $\epsilon_1=\epsilon_1(r,g,d,l)$ and $c_3=c_3(r,g,d,l)$ with the following property. For any finitely generated subgroup $\Gamma\subseteq A(\bar\QQ)$ of rank $\rho$, we have
	\[
	\#\left\{P\in X^\circ(\bar\QQ)\cap \Gamma_{\epsilon_1}': \hat h(P)> 4c_2\max\{1,h(X),\hfal(A)\}\right\}\leq c_3^{1+\rho}
	\]where $c_2=c_2(r,g,d,l)$ is taken from Theorem \ref{theorem_inequalities}.
\end{prop}
\begin{proof}
	Consider the $\rho$-dimensional real vector space $\Gamma\otimes \RR=\Gamma'\otimes \RR$ embedded in $A(\bar\QQ)\otimes \RR$, equipped with the inner product induced by the N\'eron-Tate height. 
	
	Take $c_1=c_1(r,g,d,l)$ from Theorem \ref{theorem_inequalities}. Let $\theta_1:=\frac{1}{2}[\arccos(1-\frac{1}{c_1})-\arccos(1-\frac{1}{2c_1})]$. By Lemma \ref{lemma_anglebounds} with $\theta=\theta_1$ and $c=\sqrt{c_2}$, there exists 
	\[
	\epsilon_1=\left(\delta(\theta_1)\cdot\sqrt{c_2}\right)^2=\epsilon_1(\theta_1,c_2)=\epsilon_1(c_1,c_2)=\epsilon_1(r,g,d,l),
	\]such that for any 	$v_1,v_2\in A(\bar\QQ)\otimes \RR$ with $|v_1|\geq \sqrt{c_2}$ and $|v_2-v_1|<\sqrt{\epsilon_1}$, the angle between $v_1,v_2$ is at most $\theta_1$. We need to decrease $\epsilon_1$ further soon.
	
	By Lemma \ref{lemma_cover_by_cones}, the $\rho$-dimensional vector space can be covered by at most $(1+4\sqrt{c_1})^\rho$ cones on which $\left\langle w_1,w_2 \right\rangle\geq (1-\frac{1}{2c_1})|w_1||w_2|$. Suppose $D$ is one such cone. Let $D_{\epsilon_1}$ be the $\epsilon_1$-neighborhood of $D$ in $A(\bar\QQ)\otimes \RR$ and let $D_{\epsilon_1}^+$ be the part of large points $v\in D_{\epsilon_1}$ with
	\[
	|v|^2> 4c_2\max\{1,h(X),\hfal(A)\}.
	\]
	Then for $v_1,v_2\in D_{\epsilon_1}^+$, we \emph{claim} that $\left\langle v_1,v_2\right\rangle\geq (1-\frac{1}{c_1})|v_1||v_2|$. Indeed, by definition, there is $w_1,w_2\in D$ with $|v_i-w_i|<\sqrt{\epsilon_1}$ for $i=1,2$. But
	\[
	|w_i|^2\geq (|v_i|-|v_i-w_i|)^2>(|v_i|-\sqrt{c_2})^2>c_2\max\{1,h(X),\hfal(A)\}.
	\]
	In particular $|w_i|\geq \sqrt{c_2}$. So the angle between $v_i,w_i$ is at most $\theta_1$ by the last paragraph, for $i=1,2$. But the angle between $w_1,w_2$ is at most $\arccos(1-\frac{1}{2c_1})$. Therefore by the triangle inequality, the angle between $v_1,v_2$ is at most
	$2\theta_1+\arccos(1-\frac{1}{2c_1})=\arccos(1-\frac{1}{c_1})$, hence the claim.
	
	It then suffices to bound the number of points in $X^\circ\cap D_{\epsilon_1}^+$. This is where we invoke the inequalities. In order to use Mumford's inequality, we also need to use the inductive hypothesis for lower dimensions. 
	
	Specifically, let us take any sequence of distinct points $|P_1|\leq |P_2|\leq\ldots$ in $X^\circ\cap D_{\epsilon_1}^+$ ordered by their heights. Note that we do not even know the finiteness of the sequence yet and in fact we may need to shrink $\epsilon_1$ to ensure that. Replace $\epsilon_1$ by $\min\{\epsilon_1,\epsilon(r-1,g,d^{r+2})\}$ and let $N:=c(r-1,g,d^{r+2})^{\rho+1}$, where $\epsilon,c$ are the functions in Theorem \ref{uniform_poonen} for lower dimensions. 
		
	Claim: any subset $\Xi$ of $X^\circ(\bar\QQ)\cap D_{\epsilon_1}^+$ with cardinality $\geq N+1$ is not contained in any Zariski closed subset $X'\subsetneq X$ with $\deg_L X'\leq d^{r+2}$. Indeed, if $X'\subsetneq X$ is a Zariski closed subset containing $\Xi$, then $\Xi\subseteq X'\cap X^\circ\subseteq (X')^\circ$. So  $(X')^\circ\cap \Gamma_{\epsilon_1}'$ contains $N+1$ points, which implies $\deg_L X'>d^{r+2}$.  
	
	So by Proposition \ref{prop_degree_bound_for_isolation} with
	$\Xi:=\{P_j,P_{j+1},...,P_{j+N}\}$, there is $Q_1,...,Q_r\in \Xi$ such that $(P_j,Q_1,\ldots,Q_r)$ is isolated in the fiber of $D_X:X^{r+1}\rightarrow A^r$, whence Mumford's inequality applies and we get
	\[
	|P_{j+N}|\geq |Q_i|> (1+\frac{1}{c_1})|P_j| \text{ for any }j.
	\]
	
	Take $M:=M(c_1)$ such that $(1+\frac{1}{c_1})^{M}\geq c_1$. Then
	\[
	|P_{j+NM}|>(1+\frac{1}{c_1})|P_{j+{N(M-1)}}|>\ldots>(1+\frac{1}{c_1})^M|P_j|\geq c_1|P_j|.
	\]
	Then we must have $\# X^\circ\cap D_{\epsilon_1}^+\leq rNM$, since otherwise the sequence
	\[
	P_1,P_{1+NM},P_{1+2NM},\ldots, P_{1+rNM}
	\]would contradict Vojta's inequality.
	
	Overall, we see that 
	\[
	\#\left\{P\in X^\circ(\bar\QQ)\cap \Gamma_{\epsilon_1}': \hat h(P)> 4c_2\max\{1,h(X),\hfal(A)\}\right\}\leq (1+4\sqrt{c_1})^\rho\cdot rNM.
	\]The result follows by noticing that
	\[
	(1+4\sqrt{c_1})^\rho\cdot rNM=(1+4\sqrt{c_1})^\rho\cdot r\cdot c(r-1,g,d^{2g})^{\rho+1}\cdot M(c_1)\leq c_3^{1+\rho}
	\]for some $c_3=c_3(r,g,d,l)$.

\end{proof}

In particular, we get the following qualitative result as a corollary by applying the uniform Bogomolov conjecture. This finiteness result will be used later in Proposition \ref{prop_hx_removed}. 
\begin{prop}\label{prop_finiteness}
	Assume Theorem \ref{uniform_poonen} holds for $\dim X\leq r-1$. In the case of $\dim X=r$, there exists a constant $\epsilon_0=\epsilon_0(r,g,d,l)>0$ such that for any finitely generated subgroup $\Gamma\leq A(\bar\QQ)$, the intersection  $X^\circ(\bar\QQ)\cap \Gamma'_{\epsilon_0}$ is finite.
\end{prop}
\begin{proof}
	By the uniform Bogomolov conjecture \cite[Theorem 1.3]{GGK}, there is $\epsilon=\epsilon(g,d)>0, c=c(g,d)>0$ such that for any $Q\in A(\bar\QQ)$,
	\[
	\#\{P\in X^\circ(\bar\QQ):\hat h(P-Q)\leq \epsilon\}<c.
	\]
	Let $\epsilon_0:=\frac{1}{16}\min\{\epsilon, \epsilon_1\}$. By Proposition \ref{prop_large_points}, we just need to show that the set
	\[
	\left\{P\in X^\circ(\bar\QQ)\cap \Gamma_{\epsilon_0}': \hat h(P)\leq 4c_2\max\{1,h(X),\hfal(A)\}\right\}
	\]is finite. For this, first cover the ball $B$ in $\Gamma\otimes \RR$ of radius 
	\[
	\sqrt{4c_2\max\{1,h(X),\hfal(A)\}}
	\]by finitely many balls of radius $\sqrt{\epsilon_0}$. Then the $\epsilon_0$-neighborhood $B_{\epsilon_0}$ is covered by the $\epsilon_0$-neighborhoods of the finitely many small balls. For any two points $P,Q$ in the $\epsilon_0$-neighborhood of a same small ball, we have
	\[
	|P-Q|\leq 2\sqrt{\epsilon_0}+2\sqrt{\epsilon_0}=4\sqrt{\epsilon_0}.
	\]So $\hat h(P-Q)\leq 16\epsilon_0\leq \epsilon$. Thus there are at most $c$ points in such a neighborhood. To conclude, we have finitely many regions and in each region we have finitely many points. So we get finiteness.
\end{proof}

\subsection{Removing \texorpdfstring{$h(X)$}{h(X)}}
\begin{lemma}\label{lemma_removing_hx}
	Assume $\Xi\subseteq X(\bar\QQ)$ is a finite set with the property that any equidimensional subvariety $X'\subseteq X$ of dimension $r-1$ containing $\Xi$ satisfies $\deg_L X'>ld^2/g!$. Then 
	\[
	h(X)\leq d(l/g!+1)^{r+1}\cdot \left( \max_{P\in \Xi} \hat h(P) +3\log(l/g!)  \right).
	\]
\end{lemma}
\begin{proof}
	This is \cite[Lemme 3.1]{remond-decompte} with $n=l/g!-1$.
\end{proof}

Using this lemma, we can remove $h(X)$ from Proposition \ref{prop_large_points}. The idea is to consider all translates of $X$ and find a relatively small height.
\begin{prop}\label{prop_hx_removed}
	Assume Theorem \ref{uniform_poonen} holds for $\dim X\leq r-1$. In the case of $\dim X=r$, there exist positive constants $\epsilon_2=\epsilon_2(r,g,d,l)$, $c_4=c_4(r,g,d,l)$ and $c_5=c_5(r,g,d,l)$ such that for any finitely generated subgroup $\Gamma\subseteq A(\bar\QQ)$ of rank $\rho$, either
	\[
	\# X^\circ(\bar\QQ)\cap \Gamma'_{\epsilon_2}\leq c_4^{1+\rho},
	\]or there exists $Q_0\in X^\circ(\bar\QQ)\cap \Gamma'_{\epsilon_2}$ such that 
	\[
	\#\{P\in X^\circ(\bar\QQ)\cap\Gamma'_{\epsilon_2}:\hat h(P-Q_0)> c_5\max \{1,\hfal(A)\}\}\leq c_3^{2+\rho}
	\]where $c_3=c_3(r,g,d,l)$ is taken from Proposition \ref{prop_large_points}.
\end{prop}
\begin{proof}
	Write $N:=c(r-1,g,ld^2/g!)^{2+\rho}+1$ and \[\epsilon_2:=\min\{\epsilon_0(r,g,d,l),\epsilon_1(r,g,d,l),\epsilon(r-1,g,ld^2/g!)\}\]
	where $\epsilon_0,\epsilon_1$ are from Proposition \ref{prop_finiteness} and Proposition \ref{prop_large_points} respectively. 
		
	Take any $Q\in X(\bar\QQ)$. If $P_1,\ldots,P_N$ are distinct points of $X^\circ(\bar\QQ)\cap\Gamma_{\epsilon_2}'$, then $P_1-Q,\ldots, P_N-Q$ are distinct points of $(X^\circ(\bar\QQ)-Q)\cap \left\langle\Gamma, Q\right\rangle'_{\epsilon_2}$, where $\left\langle\Gamma, Q\right\rangle$ is the subgroup of $A(\bar\QQ)$ generated by $\Gamma$ and $Q$. By Theorem \ref{uniform_poonen}, the set $\Xi:=\{P_1-Q,\ldots,P_N-Q\}$ is not contained in any subvariety $X'\subsetneq X-Q$ with $\dim X'\leq r-1$ and $\deg_L X'\leq ld^2/g!$, simply because $(X')^\circ(\bar\QQ)\cap \left\langle\Gamma, Q\right\rangle'_{\epsilon_2}$ does not contain so many points. Thus Lemma \ref{lemma_removing_hx} applies to $\Xi$ and $X-Q$ and we have
	\[
	h(X-Q)\leq d(l/g!+1)^{r+1}\cdot \left( \max_{1\leq i\leq N} \hat h(P_i-Q) +3\log(l/g!)  \right).
	\]
	
	Applying Proposition \ref{prop_large_points} to $X-Q$ and $\left\langle\Gamma,Q\right\rangle'_{\epsilon_2}$ with the above height bound, we find the cardinality of 
	\[
	\left\{P-Q\in (X^\circ(\bar\QQ)-Q)\cap \langle\Gamma,Q\rangle'_{\epsilon_2}:\hat h(P-Q)> N_1\max_{1\leq i\leq N} \hat h(P_i-Q) +N_2\max\{1,\hfal(A)\}\right\}
	\]is at most $c_3^{2+\rho}$ for some $N_1=N_1(r,g,d,l)$ and $N_2=N_2(r,g,d,l)$, which in particular implies that
	\begin{equation}\label{eqn_remove_hx}
	\#\left\{P\in X^\circ(\bar\QQ)\cap \Gamma'_{\epsilon_2}:\hat h(P-Q)> N_1\max_{1\leq i\leq N} \hat h(P_i-Q) +N_2\max\{1,\hfal(A)\}\right\}\leq c_3^{2+\rho}
	\end{equation}since $\Gamma'_{\epsilon_2}-Q$ is contained in $\left\langle\Gamma, Q\right\rangle'_{\epsilon_2}$. 
	
	Now let us restrict the choice of $Q$ in the finite set $X^\circ(\bar\QQ)\cap \Gamma'_{\epsilon_2}$. For each $Q$, there is a minimum $M=M(Q)\geq 0$ such that 
	\[
	\#\{P\in X^\circ(\bar\QQ)\cap \Gamma'_{\epsilon_2}:\hat h(P-Q)> M\}\leq c_3^{2+\rho}.
	\]
	By the finiteness, we can pick the smallest $M_0$ and assume $M_0=M(Q_0)$ for some $Q_0\in X^\circ(\bar\QQ)\cap \Gamma'_{\epsilon_2}$. We are going to show that $M_0$ is bounded by a constant multiple of $\max\{1,\hfal(A)\}$, with the constant only related to $r,g,d,l$.
		
	Consider the set $W:=\{P\in  X^\circ(\bar\QQ)\cap \Gamma'_{\epsilon_2}:\hat h(P-Q_0)\leq M_0\}$. Then $W$ is contained in the $\epsilon_2$-neighborhood of the $(1+\rho)$-dimensional ball of radius $\sqrt{M_0}$ centered at $Q_0$ in the vector space $\left\langle\Gamma,Q_0\right\rangle\otimes\RR$. In particular, by Lemma \ref{lemma_cover_by_balls}, $W$ can be covered by at most $(1+8\sqrt{N_1})^{1+\rho}$ many $\epsilon_2$-neighborhood of small balls of radius $\frac{\sqrt{M_0}}{4\sqrt{N_1}}$, centered in $\left\langle\Gamma,Q_0\right\rangle\otimes\RR$.
	 Assume that 
	\begin{equation}\label{eqn_pigeonhole}
	\#(X^\circ(\bar\QQ\cap \Gamma'_{\epsilon_2})> c_3^{2+\rho}+(N-1)\cdot (1+8\sqrt{N_1})^{1+\rho}.
	\end{equation}Then $\#W>(N-1)\cdot (1+8\sqrt{N_1})^{1+\rho}$. By the Pigeonhole principle, there exists one $\epsilon_2$-neighborhood of a small ball (call it $D_{\epsilon_2}$) that contains at least $N$ points in $W$. Assume that $P_1,...,P_N\in D_{\epsilon_2}$ are distinct. Then
	\[
	|P_i-P_1|\leq 2\cdot \frac{\sqrt{M_0}}{4\sqrt{N_1}} + 2\sqrt{\epsilon_2}
	\]for any $1\leq i\leq N$. Then by \eqref{eqn_remove_hx}, we see that
	\[
	M(P_1)\leq N_1 \cdot (2\cdot \frac{\sqrt{M_0}}{4\sqrt{N_1}} + 2\epsilon_2)^2+N_2\max\{1,\hfal(A)\}.
	\] So by our choice of $M_0$, we get
	\[
	M_0\leq N_1 \cdot 2(\frac{M_0}{4N_1}+4\epsilon_2)+N_2\max\{1,\hfal(A)\},
	\]from which we derive
	\begin{equation}\label{eqn_m0}
		M_0\leq 16N_1\cdot\epsilon_2+2N_2\max\{1,\hfal(A)\}.
	\end{equation}
	
	Finally, simply notice that the right hand side of \eqref{eqn_pigeonhole} can be bounded by $c_4^{1+\rho}$ for some $c_4=c_4(r,g,d,l)$ and the right hand side of \eqref{eqn_m0} can be bounded by $c_5\max\{1,\hfal(A)\}$ for some $c_5=c_5(r,g,d,l)$. 
\end{proof}

\medskip
\section{Small Points}\label{section_small_points}
We say an irreducible subvariety $X\subseteq A$ \emph{generates} $A$, if $X-X$ is not contained in any proper abelian subvariety of $A$. We need the following \emph{New Gap Principle} to study small points. 
\begin{theorem}\cite[Theorem 1.2]{GGK} \label{NGP}
	 For any irreducible subvariety $X\subseteq A$ that generates $A$, there exist constants $c_6=c_6(g,d)>0$ and $c_7=c_7(g,d)>0$ such that the set
	\[
	\Sigma:=\left\{P\in X^\circ(\bar\QQ):\hat h(P)\leq c_6\max\{1,\hfal(A)\}\right\}
	\]is contained in some Zariski closed subset $X'\subsetneq X$ with $\deg_L(X')<c_7$. 
\end{theorem}

\begin{cor}\label{cor_points_close}
	Assume Theorem \ref{uniform_poonen} holds for $\dim X\leq r-1$. In the case of $\dim X=r$ with $X$ generating $A$, there exist constants $\epsilon_3=\epsilon_3(r,g,d)$ and $c_8=c_8(r,g,d)$ such that for any finitely generated subgroup $\Gamma\leq A(\bar\QQ)$ of rank $\rho$ and any $Q\in A(\bar\QQ)$, we have
	\[
	\#\left\{P\in X^\circ(\bar\QQ)\cap \Gamma'_{\epsilon_3}:\hat h(P-Q)\leq c_6\max\{1,\hfal(A)\}\right\}\leq c_8^{1+\rho}.
	\]
\end{cor}
\begin{proof}
	Take $\epsilon_3:=\epsilon(r-1,g,c_7)$ and $c_8=c(r-1,g,c_7)^2$.
	Note that for any subvariety $X'\subseteq X$, we have $X^\circ\cap X'\subseteq(X')^\circ$, since the special locus of $X$ contains the special locus of $X'$ by definition. By Theorem \ref{NGP}, the set
	\[
	\Sigma_Q:=\{P\in X^\circ(\bar\QQ):\hat h(P-Q)\leq c_6\max\{1,\hfal(A)\}\}
	\]is contained in some $X'_Q\subsetneq X$ with $\deg_L(X'_Q)<c_7$. By the induction hypothesis, we have
	\[
	\#\left((X'_Q)^\circ(\bar\QQ)\cap \left\langle\Gamma, Q\right\rangle'_{\epsilon_3}\right)\leq c(r-1,g,c_7)^{2+\rho}\leq c_8^{1+\rho}
	\]for some $c_8=c_8(r,g,d)>0$. In particular, we get $\#(\Sigma_Q\cap\Gamma'_{\epsilon_3})\leq c_8^{1+\rho}$ as a subset.
\end{proof}

\medskip

\section{Proof of Theorem \ref{uniform_poonen}}\label{section_proof_of_theorem}
We will construct $\epsilon(r,g,d)$ and $c(r,g,d)$ inductively on $\dim X$. 

For $\dim X=0$, take $\epsilon(0,g,d)=\infty $ and $c(0,g,d)=d$. Then the theorem holds trivially, and the assumption in Remark \ref{remark1}(3) is satisfied.

Assume the theorem holds for $\dim X\leq r-1$. Consider the case when $\dim X=r$. Note that we can assume without loss of generality that $X$ generates $A$. Indeed, if we can prove the case when $X$ generates $A$, simply replace $A$ by the abelian subvariety $A'$ generated by $X$ and $\Gamma$ by $\Gamma\cap A'(\bar\QQ)$, so that $g$ and $\rho$ decrease, and the result follows trivially.

Let $\epsilon_2,\epsilon_3,c_3,c_4,c_5,c_6,c_8$ be as in Proposition \ref{prop_hx_removed} and Corollary \ref{cor_points_close}. Let \[\epsilon:=\min\{\epsilon_2,\epsilon_3,\frac{1}{16}c_6\}\]. Assume $\# X^\circ(\bar\QQ)\cap \Gamma'_{\epsilon}>c_4^{1+\rho}$. Then by Proposition \ref{prop_hx_removed}, there exists $Q_0\in X^\circ\cap \Gamma$ such that
\[
\#\left\{P\in X^\circ(\bar\QQ)\cap\Gamma'_{\epsilon}:\hat h(P-Q_0)> c_5\max \{1,\hfal(A)\}\right\}\leq c_3^{2+\rho}.
\]
Consider the complement $\Sigma:=\left\{P\in X^\circ(\bar\QQ)\cap \Gamma'_\epsilon :\hat h(P-Q_0)\leq c_5\max\{1,\hfal(A)\}\right\}$. Note that $\Sigma$ is contained in the $\epsilon$-neighborhood of the $(1+\rho)$-dimensional ball of radius 
\[\sqrt{c_5\max\{1,\hfal(A)\}}\]
centered at $Q_0$ in the vector space $\left\langle \Gamma,Q_0 \right\rangle\otimes \RR$. Cover $\Sigma$ using $\epsilon$-neighborhoods of small balls of radius
$ \frac{1}{4}\sqrt{c_6\max\{1,\hfal(A)\}}$ centered in $\left\langle \Gamma,Q_0 \right\rangle\otimes \RR$.
By Lemma \ref{lemma_cover_by_balls}, $\Sigma$ can be covered by $(1+8\sqrt{\frac{c_5}{c_6}})^{1+\rho}$ such neighborhoods of the small balls.  For $P_1,P_2$ in a same $\epsilon$-neighborhood, we have
\[
\hat h(P_1-P_2)\leq \left(2\sqrt{\epsilon}+2\cdot \frac{1}{4}\sqrt{c_6\max\{1,\hfal(A)\}}\right)^2\leq  c_6\max\{1,\hfal(A)\}.
\]Hence by Corollary \ref{cor_points_close}, there are at most $c_8^{1+\rho}$ points of $X^\circ\cap\Gamma'_{\epsilon}$ in one $\epsilon$-neighborhood. So
\[
\#(X^\circ(\bar\QQ)\cap\Gamma'_{\epsilon})\leq \#\Sigma +c_3^{2+\rho} \leq \left(1+8\sqrt{\frac{c_5}{c_6}}\right)^{1+\rho}\cdot c_8^{1+\rho}+c_3^{2+\rho}\leq c^{1+\rho}
\]for some $c=c(r,g,d,l)\geq c_4$.

Finally since $X$ generates $A$, the degree $l$ of $A$ is actually bounded by a function of the degree $d$ of $X$ and dimension $g$ of $A$, see \cite[\S2]{GGK}. So we can remove the dependence on $l$ and we are done.

\medskip

\section{Finiteness of Cosets}\label{section_finiteness_of_cosets}
In this section, we show that Theorem \ref{uniform_poonen} can be improved to Theorem \ref{uniform_poonen2}, to include the counting of the positive dimensional cosets. The idea is the same as \cite[Lemma 10.4]{Gao_survey}. Basically, we need to bound the degrees of subvarieties in the special locus and use induction.

\begin{proof}[Proof of Theorem \ref{uniform_poonen2}]
	 Without loss of generality, assume $X$ generates $A$. Let $\Sigma(X)$ be the set of positive dimensional abelian subvarieties $B\subseteq A$ such that there is $x\in X(\bar\QQ)$ satisfying $x+B\subseteq X$, and $B$ is maximal for $x$. Bogomolov \cite[Theorem 1]{Bogomolov_1981} showes that there is an upper bound $\delta_1=\delta_1(g,d)$ for the degree of $B\in \Sigma(X)$. R\'emond \cite[Proposition 4.1]{remond-decompte} proves that there is $N_1:=N_1(g,l,\delta_1(g,d))$ such that $\#\Sigma(X)\leq N_1$.
	 
	 The key idea is to take a complement $B^\perp$ of $B$, such that $B+B^\perp=A$ and $B\cap B^\perp$ is finite. It is possible to choose such a $B^\perp$ with degree at most $\delta_2=\delta_2(g,d,l)$, see \cite{MasWus}. The $B^\perp$ will serve as a substitute for $A/B$. Write $(X:B):=\{x\in X: x+B\subseteq X\}$. Note that (recall $r=\dim X$)
	 \[
	 (X:B)=\bigcap_{b\in B}(X+b)=\bigcap_{i=0}^{r} (X+b_i)
	 \]if we choose $b_0,\ldots,b_r\in B$ in a general position, for dimension reason. Then we let $X_B:=(X:B)\cap B^\perp.$ We have $X_B+B=(X:B)$. By B\'ezout's theorem, the degree of $X_B$ is bounded by $d^{r+1}\cdot\delta_2\leq \delta_3=\delta_3(g,d,l)$. Note also that $\Sp(X)$ can be written as a union 
	 \[
	 \Sp(X)=\bigcup_{B\in \Sigma(X)}(X:B)=\bigcup_{B\in \Sigma(X)}(X_B^\circ+B),
	 \] where the second inequality uses the fact that if $x+B'\subseteq X_B$, then \[x+(B+B')\subseteq (X:(B+B')).\]
	 
	 Take any finitely generated subgroup $\Gamma\leq A(\bar\QQ)$ of rank $\rho$. For each $B\in \Sigma(X)$, we define $\Gamma_B\subseteq B^\perp(\bar\QQ)$ to be the pullback of $\Gamma+B/B$ under the isogeny $B^\perp\rightarrow A/B$. Then $\Gamma_B$ is of the same rank. Note that $(X_B^\circ+B)(\bar\QQ)\cap\Gamma'_{\epsilon}\subseteq (X_B^\circ(\bar\QQ)\cap \Gamma_{B,\epsilon}')+B$.
	 
	 Applying Theorem \ref{uniform_poonen} to $X_B$ and $\Gamma_B$, we get $\epsilon=\epsilon(g,\delta_3), c=c(g,\delta_3)$ such that (recall that we choose the constants in a way that they also work for reducible varieties) $\#(X_B^\circ(\bar\QQ)\cap \Gamma'_{B,\epsilon})\leq c^{1+\rho}.$ Then we have
	 \[
	 \Sp(X)(\bar\QQ)\cap \Gamma'_\epsilon\subseteq\bigcup_{B\in \Sigma(X)}\left[(X_B^\circ(\bar\QQ)\cap \Gamma_{B,\epsilon}')+B\right],
	 \] where the right hand side is the union of at most $\#\Sigma(X)\cdot c^{1+\rho}\leq N_1\cdot c^{1+\rho}$ cosets. Since $X$ generates $A$, we could bound $l$ in terms of $g$ and $d$ and hence remove $l$ in a trivial way. So we are done by simply combining the above result on $\Sp(X)$ with the result on $X^\circ$.
	 \end{proof}
	 
\medskip	 

\section{Further comments}\label{section_further comments}
In Theorem \ref{uniform_poonen}, $\epsilon$ is only related to the dimension of the abelian variety $A$ and the degree of the subvariety $X$. On the other hand, as suggested by the New Gap Principle \ref{NGP}, the "generic" distance between two points on $X$ is proportional to $\max\{1,\hfal(B)\}$, where $B$ is the abelian subvariety generated by $X$ (recall from \S\ref{section_small_points} that this means $X-X$ is not contained in any proper abelian subvariety of $A$). The exact same method (except in Proposition \ref{prop_finiteness}, one needs to invoke the New Gap Principle \ref{NGP}) can be used to show that Theorem \ref{uniform_poonen} is true with $\epsilon$ replaced by 
\[
\epsilon\cdot\max\{1,\inf_{B\subseteq A}\hfal(B)\},
\]where the infimum is taken over all positive-dimensional abelian subvarieties $B$ of $A$. The method does not work without taking the infimum above, since one has no control over $X'$ in the New Gap Principle \eqref{NGP}.

We cannot in general hope $\epsilon$ to be replaced by an even stronger form $\epsilon\cdot\max\{1,\hfal(A)\}$. A counterexample may be easily constructed: consider $X\times \{0\}\subseteq A\times B$ with $A$ fixed and $B$ varying of the same dimension. Then the degree of $X\times\{0\}$ and the dimension of $A\times B$ are fixed, but the Faltings height of $A\times B$ has no bound. A possible way to get around is to only consider the points that are ``transverse'', as suggested by the referee. \Tangli{Thanks for the comments! I rewrite this paragraph.}

We might also consider the set
\[
\{x\in X^\circ(\bar\QQ):d(x,\Gamma\otimes\RR)\leq \alpha|x|+\beta\}
\]with $\alpha,\beta$ positive constants and $d(x,\Gamma\otimes\RR)$ denoting the distance from $x$ to any $\RR$-linear combination of vectors of $\Gamma$, as in \cite[Theorem 1.3]{Zhang00}. What we showed is the existence of a uniform $\beta$ with $\alpha=0$ to make this set uniformly bounded in terms of the rank. It would be interesting to investigate whether we can pick positive $\alpha,\beta$ uniformly.

\section*{Acknowledgement}
I would like to thank my advisor Dan Abramovich for encouraging me to write the paper and thank Ziyang Gao for many helpful conversations. I thank Niki Myrto Mavraki for bringing Zhang's paper to my attention. I thank the referees for suggesting various improvements.\Tangli{Thank you for the help!}

\bibliographystyle{plain}
\bibliography{Arithmeticgeometry}
\end{document}